\documentclass[11pt]{amsart}

\usepackage{enumerate,url,amssymb,  mathrsfs}

\newtheorem{theorem}{Theorem}[section]

\newtheorem*{lemma*}{Lemma}
\newtheorem{proposition}[theorem]{Proposition}
\newtheorem{corollary}[theorem]{Corollary}

\theoremstyle{definition}
\newtheorem{definition}[theorem]{Definition}
\newtheorem{example}[theorem]{Example}

\newtheorem{question}[theorem]{Question}

\theoremstyle{remark}
\newtheorem{remark}[theorem]{Remark}

\numberwithin{equation}{section}

\newcommand{\abs}[1]{\lvert#1\rvert}
\newcommand{\norm}[1]{\lVert#1\rVert}

\newcommand{\A}{\mathbb{A}}
\newcommand{\C}{\mathbb{C}}
\newcommand{\Q}{\mathbb{Q}}
\newcommand{\U}{\mathbb{U}}
\newcommand{\D}{\mathbb{D}}
\newcommand{\E}{\mathcal{E}}

\newcommand{\R}{\mathbb{R}}
\newcommand{\T}{\mathbb{T}}

\newcommand{\Ar}{\mathscr{A}}
\newcommand{\W}{\mathscr{W}}

\newcommand{\dtext}{\textnormal d}
\newcommand{\onto}{\stackrel{{\rm \tiny{onto}}}{\longrightarrow}}
\DeclareMathOperator{\diam}{diam}

\DeclareMathOperator{\re}{Re}
\DeclareMathOperator{\im}{Im}
\DeclareMathOperator{\id}{id}

\DeclareMathOperator{\loc}{loc}

\def\XXint#1#2#3{{\setbox0=\hbox{$#1{#2#3}{\int}$}
\vcenter{\hbox{$#2#3$}}\kern-.5\wd0}}

\def\le{\leqslant}
\def\ge{\geqslant}
\begin{document}

\title[Hopf differentials and smoothing Sobolev homeomorphisms]{Hopf differentials and \\ smoothing Sobolev homeomorphisms}

\author{Tadeusz Iwaniec}
\address{Department of Mathematics, Syracuse University, Syracuse, NY 13244, USA
and Department of Mathematics and Statistics, University of Helsinki, Finland}
\email{tiwaniec@syr.edu}
\thanks{Iwaniec was supported by the NSF grant DMS-0800416.}

\author{Leonid V. Kovalev}
\address{Department of Mathematics, Syracuse University, Syracuse,
NY 13244, USA}
\email{lvkovale@syr.edu}
\thanks{Kovalev was supported by the NSF grant DMS-0968756.}

\author{Jani Onninen}
\address{Department of Mathematics, Syracuse University, Syracuse,
NY 13244, USA}
\email{jkonnine@syr.edu}
\thanks{Onninen was supported by the NSF grant  DMS-1001620.}


\date{June 26, 2010}


\begin{abstract} 
We prove that planar homeomorphisms can be approximated by diffeomorphisms 
in the Sobolev space $\W^{1,2}$ and in the Royden algebra. 
As an application, we show that every  discrete and open planar mapping with a holomorphic Hopf differential is harmonic. 
\end{abstract}
\maketitle
\section{Introduction}

It is a fundamental property of Sobolev spaces $\W^{1,p}$, $1\le p<\infty$, that any element can be approximated 
strongly (i.e., in the norm) by $\mathscr C^{\infty}$ smooth functions, or by piecewise affine ones. 
In the context of vector-valued Sobolev functions, that is, mappings in
$\W^{1,p}(\Omega,\R^n)$, invertibility comes into play. Indeed, the studies of invertible Sobolev mappings are 
of great importance in nonlinear elasticity \cite{Ba0, FG, MST, Sv}.  The following 
natural question was put forward by John M. Ball.

\begin{question}\cite{Ba}
If $u \in \W^{1,p}(\Omega, \R^n)$ is invertible, can $u$ be approximated in $ \W^{1,p}$ by piecewise affine invertible mappings?
\end{question}

J. Ball attributes this question to L.C. Evans, who was led to it through his investigation of the  
partial regularity of minimizers~\cite{Ev} of neohookean energy functionals~\cite{Ba1, BPO1, CL, SiSp}. We provide an affirmative solution of the Ball-Evans problem in the case $p=n=2$. The most general formulation of our result administers Royden algebras $\Ar(\Omega)$ and $\Ar_\circ (\Omega)$, see Section~\ref{secroy}. We write
\[\E[h]=\E_\Omega[h]:= \norm{Dh}^2_{\mathscr L^2(\Omega)}= \int_\Omega \abs{Dh(z)}^2\, \dtext z \]
where $\abs{Dh}$ is the Hilbert-Schmidt norm of the differential.

\begin{theorem}[Approximation by diffeomorphisms]\label{thmapprox} Let $h \colon \Omega \onto \Omega^\ast$ be a homeomorphism of Sobolev class $\mathscr W^{1,2}_{\loc} (\Omega, \Omega^\ast)$. Then for every $\epsilon >0$ there exist a diffeomorphism $H \colon \Omega\onto \Omega^\ast$ such that
\begin{enumerate}[(i)]
\item $H-h \in \Ar_\circ (\Omega)$
\item $ \norm{H-h}_{\Ar (\Omega)}\le \epsilon$
\item $\E[H] \le \E[h]$.
\end{enumerate}
\end{theorem}

Part (iii) is nontirivial only in the finite energy case, $\E_\Omega [h]< \infty$. 
Let us note that the existence of smooth approximation implies the existence of piecewise-affine approximation, since a diffeomorphism can be triangulated. (In the converse direction, a piecewise-affine mapping can be smoothed in 
dimensions less than four~\cite{Mu}, but not in general.)  
Partial results toward the Ball-Evans problem were obtained in~\cite{Mo} (for planar bi-Sobolev mappings that are smooth outside of a finite set) and in~\cite{BM} (for planar bi-H\"older mappings, with approximation in the H\"older norm). 
The articles~\cite{Ba,SS} illustrate the difficulty of preserving invertibility in the process of smoothing a Sobolev homeomorphism.  

We also give an application of Theorem~\ref{thmapprox} to a problem that originated in a series of papers 
by Eells, Lemaire and Sealey~\cite{EL2, EL3, Se}. It concerns the nonlinear differential equation 
\begin{equation}\label{heq2}
\frac{\partial}{ \partial \bar z} \left( h_z \overline{h_{\bar z}} \right)=0
\end{equation}
for mappings defined in a domain in the complex plane $\C$. Naturally, the Sobolev space $\W^{1,2}_{\loc}(\Omega, \C)$ should be considered as the domain of definition of equation~\eqref{heq2}. This places $ h_z \overline{h_{\bar z}}$ in $\mathscr L^1_{\loc} (\Omega)$, so the complex Cauchy-Riemann derivative $\frac{\partial}{\partial \bar z}$ applies in the sense of distribution. By Weyl's lemma $h_z \overline{h_{\bar z}}$ is a holomorphic function.

The expression $Q_h:=h_z \overline{h_{\bar z}}\, \dtext z \otimes \dtext z$ is known as the Hopf differential of $h$ (named after H. Hopf, who employed a similar device, e.g., in \cite[Chapter VI]{Ho}).
It is clear that $Q_h$ is a holomorphic quadratic differential whenever $h$ is harmonic, which is a general fact about energy-stationary  mappings between Riemannian manifolds~\cite[(10.5)]{EL1},~\cite{Job} and~\cite{Stb}. Eells and Lemaire inquired about the possibility of a converse result, e.g., for mappings with finite energy and almost-everywhere positive Jacobian~\cite[(2.6)]{EL2}. In this setting a counterexample was provided by Jost~\cite{Jo}, who also proved the existence of $\W^{1,2}$-solutions of~\eqref{heq2} in every homotopy class of mappings between compact Riemann surfaces. A more restricted form of the Eells-Lemaire problem, \cite[(5.11)]{EL3} and~\cite{Se}, imposed the additional assumption that $h$ is a quasiconformal homeomorphism, and was settled by H\'elein~\cite{He} in the affirmative. Here we dispose with the quasiconformality condition and treat general planar homeomorphisms of finite energy. Since the inverse of such a homeomorphism need not be in any Sobolev class~\cite{HK}, some difficulties are to be expected. They shall be overcome with the aid of our approximation theorem~\ref{thmapprox}.

\begin{theorem}\label{thmhopf}
Every continuous, discrete and open mapping $h$ of Sobolev class  $\W^{1,2}_{\loc}(\Omega, \C)$ that satisfies equation~\eqref{heq2} is  harmonic. 
\end{theorem}

The failure of Theorem~\ref{thmhopf} for uniform limits of homeomorphisms should be mentioned. This 
is illustrated by Example~\ref{ex}.

\section{Background}\label{secroy}
Let $\Omega$ be a bounded domain in $\R^2 \simeq \C$, nonempty open connected set. We consider a class $\Ar (\Omega)$ of uniformly continuous functions $h \colon {\Omega} \to \C$ having finite Dirichlet energy, and furnish it with the norm
\[\| h\|_{\Ar (\Omega)} = \| h \|_{\mathscr C ({\Omega})} + \| Dh \|_{\mathscr L^2 (\Omega)} < \infty  \]
$\Ar (\Omega)$ is a commutative Banach algebra with the usual multiplication of functions in which $\norm{h_1h_2}_{\Ar (\Omega)} \le \norm{h_1}_{\Ar (\Omega)} \norm{h_2}_{\Ar (\Omega)}$. The closure of $\mathscr C^\infty_\circ (\Omega)$ in $\Ar (\Omega)$ will be denoted by $\Ar_\circ (\Omega)$. Suppose, to look at more specific situation, that $\Omega = \U$ is a Jordan domain; that is, a simply connected open set whose boundary $\Gamma= \partial \U$ is a closed Jordan curve. By goodness of the Carath\'eodory extension theorem~\cite[p. 18]{Pob}, there is a homeomorphism $\varphi \colon \overline{\mathbb D} \stackrel{\textrm{\tiny{onto}}}{\longrightarrow} \overline{\U}$ of the closed unit disk $\overline{\D}= \{ \xi \colon \abs{\xi} \le 1 \}$ that is conformal in $\D$. After the change of variable, $z= \varphi (\xi)$, we obtain a function $H(\xi)= h\big(\varphi (\xi)\big)$ in $\Ar (\D)$. The operation
\[\mathbf{T}_\varphi \colon \Ar (\U) \to \Ar (\D)\]
so defined is an isometry; $\norm{\mathbf{T}_\varphi h}_{\Ar (\D)} = \norm{h}_{\Ar (\U)}$. Furthermore,
\[\mathbf{T}_\varphi \colon \Ar_\circ (\U) \to \Ar_\circ (\D)\]
\begin{proposition}[A generalization of Poisson's formula]\label{prpoisson}
Let $\U$ be a Jordan domain. There is (unique) bounded linear operator
\[\mathbf{P}_{\U} \colon \Ar(\U) \to \Ar(\U)\]
such that
\[
\begin{cases} \mathbf{P}_{\U} - \mathbf Id\colon \Ar (\U) \to \Ar_\circ (\U)\\
\Delta \circ \mathbf P_{\U}=0
\end{cases}
\]
\end{proposition}
We name $\mathbf P_{\U}$ the {\it Poisson operator}. The energy of $\mathbf P_{\U} h$ does not exceed that of $h$. This fact is known as  {\it Dirichlet's principle}
\[\int_{\U} \abs{D \mathbf P_{\U} h}^2 \le \int_{\U} \abs{D h}^2\]
The proof of this proposition reduces to the case when $\U=\D$, by conformal change of variables. A routine verification of this case  is left to the reader. We only indicate that the less familiar property $\mathbf{P}_{\D} h -h \in \Ar_\circ (\D)$, for $h\in \Ar (\D)$, needs to be justified.
\begin{corollary}[Harmonic replacement]
Let $\Omega$ be  a domain in $\C$ and $\U \subset \overline{\U} \subset \Omega$ a Jordan domain. There exists (unique) bounded linear operator
\[\mathbf R_{\U} \colon \Ar (\Omega) \to \Ar (\Omega)\]
such that, for every $h\in \Ar (\Omega)$
\[\begin{cases} \mathbf R_{\U}h=h \qquad \mbox{ on }\quad \Omega \setminus \U \\
\Delta  \mathbf R_{\U}h= 0  \qquad \mbox{ in } \quad \U
\end{cases}\]
The Laplace equation yields $\E_{\Omega} [\mathbf R_{\U} h] \le \E_{\Omega} [h]$. Equality occurs if and only if $h$ is harmonic in $\U$.
\end{corollary}
A short proof of this corollary runs somewhat as follows. The unique harmonic extension of $h \colon \partial \U \to \C$ inside $\U$  given by $\mathbf P_{\U} h$  has the property that $\mathbf P_{\U} h-h \in \Ar_\circ (\U)$. Therefore, the zero extension of $\mathbf P_{\U} h -h$ outside $\U$, denoted by $[\mathbf P_{\U}h-h]_\circ$, belongs to $\Ar(\Omega)$. We define
\[\mathbf{R}_{\U} h:= [\mathbf P_{\U} h -h]_\circ +h  \in \Ar (\Omega)\]
The desired properties of the operator $\mathbf R_{\U}$ so defined are automatically fulfilled in view of Proposition~\ref{prpoisson}.
\begin{proposition}\label{RKC}
Let $\Omega$ be  a domain in $\C$ and $\U \subset \overline{\U} \subset \Omega$ a Jordan domain. Suppose that $h\in \mathscr \Ar (\Omega)$ is a homeomorphism of $\Omega$ onto  $h(\Omega)$ and $h(\U)$ is convex. Then
$\mathbf R_\U h$ is homeomorphism in $\Omega$ and is a harmonic diffeomorphism in $\U$.
\end{proposition}
The injectivity of $\mathbf R_\U h$ is the content of the Rad\'o-Kneser-Choquet Theorem~\cite[p. 29]{Dub}.  Furthermore, planar harmonic homeomorphisms are $\mathscr C^\infty$-smooth diffeomorphisms according to Lewy's theorem~\cite[p. 20]{Dub}.

\section{Smoothing Sobolev homeomorphisms, Theorem~\ref{thmapprox}}
\begin{proof}[Proof of Theorem~\ref{thmapprox}]
We may and do assume that $h$ is not harmonic, since otherwise $H =h$ satisfies the desired properties by Lewy's theorem (mentioned above). Let $z_\circ \in \Omega$ be a point such that $h$ fails to be harmonic in any neighborhood of $z_\circ$. By choosing the origin of the coordinate system we ensure that $h(z_\circ)$ does not lie on the boundary of any dyadic squares associated with the coordinate system. 

Let us choose and fix any $\epsilon >0$. The construction of $H$ proceeds in 5 steps. We construct homeomorphisms $h_k \colon \Omega \onto \Omega^\ast$, $k=0, \dots, 5$ such that $h_0=h$, $h_k\in h_{k-1}+\Ar_\circ(\Omega)$, $h_5$ is a diffeomorphism, and $\norm{h_k-h_{k-1}}$ is bounded by a multiple of $\epsilon$ for each $k$. In each step we modify the previous construction to gain better regularity. In steps 1, 2 and 4 we use harmonic replacement according to Proposition~\ref{RKC}. 
In steps 3 and 5 we smoothen the mapping near the boundaries of the domains in which harmonic replacement was performed.  
The result of each step is denoted by $h_1,\dots,h_5$. The \textsf{finite energy case} $h\in \Ar(\Omega)$ requires a few additional details, which are provided at the end of each step. 

We begin with a decomposition of the target domain
\begin{equation}
\Omega^\ast = \bigcup_{\nu=1}^\infty \overline{\Q_{\nu}}
\end{equation}
into closed nonoverlapping dyadic squares $\overline{\Q_{\nu}} \subset \Omega^\ast$.
This decomposition is made by selecting the maximal dyadic squares that lie  in $\Omega^\ast$. Thus the cover of $\Omega^\ast$ by such squares is locally finite. The preimage of $\Q_\nu$ under $h$, denoted by $\U_\nu$, is a Jordan domain in $\Omega$. Hereafter $\U_\nu$ will be referred to as the curved-square. In fact to every partion of $\Omega^\ast$ into closed squares there will correspond a partition of $\Omega$ into closed curved-squares via the mapping $h \colon \Omega \stackrel{\textrm{\tiny{onto}}}{\longrightarrow} \Omega^\ast$, for example:
\[\Omega = \bigcup_{\nu=1}^\infty \overline{\U_{\nu}}\]
{\bf Step 1.} For each $\U_\nu$ we replace $h \colon \overline{\U_\nu} \stackrel{\textrm{\tiny{onto}}}{\longrightarrow} \overline{\Q_{\nu}}$ with a piecewise harmonic homeomorphism $h_1 \colon \overline{\U_\nu} \stackrel{\textrm{\tiny{onto}}}{\longrightarrow} \overline{\Q_{\nu}} $ that coincides with $h$ on $\partial  \overline{\U_{\nu}}$. To this effect we partition the square $ \overline{\Q_{\nu}} $,
\begin{equation}
 \overline{\Q_{\nu}} =  \overline{\Q^1_{\nu}} \cup  \overline{\Q^2_{\nu}} \cup \dots \cup  \overline{\Q^n_{\nu}}, \qquad (n=n_\nu=4^{k_\nu})
\end{equation}
into congruent dyadic squares $ \overline{\Q^i_{\nu}}$, $i=1, \dots, n$. The number $n$, depending on $\nu$, will be determined later. For the moment fix $\nu$ and look at the homeomorphisms
\[h \colon  \overline{\U^i_\nu} = h^{-1}(\overline{\Q^i_{\nu}}) \stackrel{\textrm{\tiny{onto}}}{\longrightarrow} \overline{\Q^i_{\nu}}
\]
These mappings belong to the Royden algebra $\Ar ( \overline{\U^i_\nu})$. With the aid of Propositions~\ref{prpoisson} and~\ref{RKC} we replace each    $h \colon \overline{\U^i_\nu} \to  \overline{\Q^i_\nu}$ with a harmonic homeomorphism $h^i_\nu \colon \overline{\U^i_\nu} \stackrel{\textrm{\tiny{onto}}}{\longrightarrow}   \overline{\Q^i_\nu}$ that coincides with $h$ on $\partial  \overline{\U^i_\nu}$, $i=1,2, \dots , n$. Such mappings are $\mathscr C^\infty$-smooth diffeomorphisms  $h^i_\nu \colon {\U^i_\nu} \stackrel{\textrm{\tiny{onto}}}{\longrightarrow}  {\Q^i_\nu}$. Moreover, $h^i_\nu -h \in \Ar_\circ (\U_\nu^i)$ and
\begin{equation}
\begin{cases}
\E_{\U^i_\nu}[h_\nu^i] \le \E_{\U^i_\nu}[h] \qquad \mbox{for } 1,2, \dots , n \\
 \E_{\partial \U^i_\nu}[h_\nu^i] = \E_{\partial \U^i_\nu}[h], \quad \mbox{ because } h_\nu^i=h \mbox{ on } \partial \U_\nu^i
\end{cases}
\end{equation}
We  obtain a piecewise harmonic homeomorphism by gluing $h^i_\nu$  together along the common boundaries of $\U^i_\nu$. Denote it by
\[h^n_{\nu} \colon  \overline{\U_\nu}  \stackrel{\textrm{\tiny{onto}}}{\longrightarrow} \overline{\Q_{\nu}} \]
\[h^n_{\nu} \in h + \Ar_\circ (\U_\nu)\]
Precisely we define
\[h_\nu^n=h+ \sum_{i=1}^n [h_\nu^i-h]_\circ\]
Here and in the sequel the notation $[\varphi]_\circ$ for $\varphi\in \Ar_\circ(\U)$ stands for zero extension of $\varphi$ to the entire domain $\Omega$. Obviously $[\varphi]_\circ \in \Ar_\circ(\Omega)$.  The above construction depends on the number $n$.  For $\nu$ fixed we actually have a sequence $\{h^n_{\nu}\}_{n=1,2, \dots}$ that is bounded in $\Ar (\U_\nu)$. However, we have uniform bounds independent of $n$,
\[\norm{h^n_{\nu}}_{\mathscr C(\overline{\U_\nu})} \le \diam \Q_\nu \]
and
\[\E_{\U_\nu} [h^n_{\nu}] \le \E_{\U_\nu} [h]\]
The key observation is that
\begin{equation}
\begin{cases}
h^n_{\nu}-h \in \Ar_\circ (\U_\nu)\\
\lim\limits_{n \to \infty} \norm{h^n_{\nu}-h}_{\Ar (\U_\nu)}=0
\end{cases}
\end{equation}
Indeed, for $z\in \overline{\U^i_\nu}$ we have
\[\abs{h^n_{\nu}(z)-h(z)} \le \diam \Q_\nu^i = \frac{1}{\sqrt{n}} \diam \Q_\nu\]
Thus $h^n_{\nu } \rightrightarrows h$ uniformly on $\overline{\U_\nu}$ as $n \to \infty$. On the other hand the differential matrices $Dh^n_{\nu}$ are bounded in $\mathscr L^2 (\U_\nu, \R^{2\times 2})$. Their weak limit exits and is  exactly equal to $Dh$, because the mappings converge uniformly to $h$. Therefore,
\[
\begin{split}
\int_{\U_\nu} \abs{Dh^n_{\nu}-Dh}^2 & = \int_{\U_\nu} \Big(\abs{Dh^n_{\nu}}^2 + \abs{Dh}^2-2 \langle Dh^n_{\nu}, Dh \rangle  \Big)\\
& \le  2 \int_{\U_\nu} \Big(\abs{Dh}^2 -   \langle Dh^n_{\nu}, Dh \rangle  \Big) \\
& = 2 \int_{\U_\nu} \langle Dh, Dh - Dh^n_{\nu} \rangle \longrightarrow 0
\end{split}
\]
We can now determine the number $n=n_\nu$ of congruent dyadic squares in $\Q_\nu$, simply requiring that
\[
\begin{cases}
 \diam \Q^i_\nu \le {\epsilon} \qquad \mbox{ for every } i=1,2, \dots, n_\nu\\
\norm{Dh^n_{\nu}-Dh}_{\mathscr L^2 (\overline{\U_\nu})} \le \epsilon \cdot 2^{-\nu}
\end{cases}
\]
Fix such $n=n_\nu$ and abbreviate the notation for $h^{n_\nu}_{\nu}$ to $h^\nu$. We  obtain a homeomorphism
\[h_1:=h+ \sum_{\nu =1}^\infty [h^\nu-h]_\circ \in h+ \Ar_\circ (\Omega)\]
where we recall that $ [h^\nu-h]_\circ$ stands for the zero extension of $h^\nu-h$ to the entire domain $\Omega$. Clearly, $h_1$ is harmonic in each $\U_\nu^i$, $\nu=1,2, \dots$, $i=1,2, \dots , n_\nu$ and we have
\[\norm{h_1-h}_{\mathscr C (\Omega)} \le \sup \{\diam \Q_\nu^i \colon \nu=1,2, \dots, \; \; i=1, \dots, n_\nu \} < {\epsilon}\]
\begin{equation}
\norm{h_1-h}_{\Ar(\Omega)} \le {\epsilon}+\sum_{\nu =1}^\infty  \norm{Dh^\nu-Dh}_{\mathscr L^2(\U_\nu)} \le {\epsilon}+ \sum_{\nu=1}^\infty \epsilon \cdot 2^{-\nu} = 2 \epsilon
\end{equation}
For further considerations it will be convenient to number the squares $\Q_\nu^i$ and their preimages $\U_\nu^i$ using only one index. These sets will be respectively denoted by $\Q^\alpha$ and $\U^\alpha$, $\alpha=1,2, \dots$. For the record,
\begin{equation}\label{doeq}
\diam \Q^\alpha \le {\epsilon}, \qquad \alpha =1,2, \dots
\end{equation}
\textsf{Finite energy case.} Summing up the energy inequalities for the mappings $h^i_\nu \colon \U_\nu^i \to \Q_\nu^i$ we see that the total energy of $h_1$ does not exceed the energy of $h$. Even more, since $h$ was assumed to be not harmonic, there is at least one region $\U_\nu^i$ for which $h \colon \U_\nu^i \to \Q_\nu^i$ was not harmonic. Consequently, its harmonic replacement results in strictly smaller energy. Hence
\begin{equation}\label{36p}\E_{\Omega} [h_1] < \E_{\Omega} [h], \qquad \mbox{so let } \delta= \norm{Dh}_{\mathscr L^2(\Omega)} - \norm{Dh_1}_{\mathscr L^2(\Omega)} >0\end{equation} 

{\bf Step 2.}  Denote by $\mathcal F=\{\Q^\alpha \colon \alpha=1,2, \dots\}$ the family of all open squares $\Q^\alpha \subset \overline{\Q^\alpha} \subset \Omega^\ast$ that are build in Step 1 for the construction of the  mapping $h_1 \colon \Omega \to \Omega^\ast$. Let $\mathcal V$ be the set of vertices of these squares. Whenever two squares $\Q^\alpha, \Q^\beta \in \mathcal F$, $\alpha \ne \beta$, meet along their boundaries the intersection $I^{\alpha , \beta}= \partial \Q^\alpha \cap \partial \Q^\beta$ is either a point in $\mathcal V$ or a closed interval with endpoints in $\mathcal V$. Denote by $\mathcal J \subset \{I^{\alpha , \beta} \colon \alpha\ne \beta,\;  \alpha, \beta=1,2, \dots \}$ the subfamily of all such intersections, excluding empty set and  vertices. For each interval $I^{\alpha , \beta} \in \mathcal J$ we shall construct a doubly convex lens-shaped region $\mathbb L^{\alpha, \beta}$ with $I^{\alpha, \beta}$ as its axis of symmetry in the following way.  Let $R$ be a number greater than the length of $I^{\alpha, \beta}$ to be chosen later. There exist exactly two open disks of radius $R$ for  which $I^{\alpha, \beta}$ is a chord. Let $\mathbb L_R^{\alpha, \beta}$ be their intersection. This is a symmetric doubly convex lens of curvature $\frac{1}{R}$. Thus  $\mathbb L_R^{\alpha, \beta}$ is bounded by two circular arcs $\gamma^{\alpha, \beta}= \Q^\alpha \cap \partial \mathbb L_R^{\alpha, \beta}$ and $\gamma^{\beta, \alpha}=\Q^\beta \cap \partial \mathbb L_R^{\alpha, \beta} $.  As the curvature of the lens approaches zero the area of $\mathbb L_R^{\alpha, \beta}$ tends to $0$. This allows us to choose $R$ depending on $\alpha$ and $\beta$ so that the lenses $\mathbb L^{\alpha, \beta}=\mathbb L_R^{\alpha, \beta}$ have the following property.
\begin{equation}\label{shouldnumber}
\int_{\mathbb K^{\alpha, \beta}} \abs{Dh_1}^2 < \frac{\epsilon^2}{2^{\alpha+\beta}}, \quad \mbox{ where } \mathbb K^{\alpha, \beta}=h^{-1}_1(\mathbb L_R^{\alpha, \beta})
\end{equation}
The lenses $\mathbb L^{\alpha, \beta}$ are disjoint because the opening angle of each  lens is at most $\pi/3$ and their axes are either parallel or orthogonal.  However, the closures of the lenses considered here  may have a common point that lies in $\mathcal V$. On each  $\mathbb K^{\alpha, \beta}$ we replace $h_1$ by the harmonic extension of its restriction to $\partial \mathbb K^{\alpha, \beta}$. Thus we obtain a homeomorphism  $h_2^{\alpha, \beta} \colon \overline{\mathbb K^{\alpha, \beta}} \onto \mathbb \overline{L^{\alpha, \beta}}$ of class $h_1+ \Ar_\circ (\mathbb K^{\alpha, \beta})$. By Proposition~\ref{RKC} the mappings
 $h_2^{\alpha, \beta} \colon \mathbb K^{\alpha, \beta} \onto \mathbb L^{\alpha, \beta}$ are  diffeomorphisms.
 Finally, we define
 \[h_2=h_1+\sum_{\alpha, \beta} [h_2^{\alpha, \beta}-h_1]_\circ \in h_1+ \Ar_\circ (\Omega) = h+ \Ar_\circ (\Omega) \]
 and observe that, from~\eqref{doeq},
 \[\norm{h_2-h_1}_{\mathscr C (\Omega)}\le \sup_{\alpha, \beta} \diam \left( \mathbb L^{\alpha, \beta} \right) \le \epsilon.  \]
Also, ~\eqref{shouldnumber} and Dirichlet's principle imply
 \begin{equation*}
\begin{split}
 \int_{\Omega} \abs{Dh_2-Dh_1}^2 &\le \sum_{\alpha, \beta} \int_{\mathbb K^{\alpha, \beta}}  2\left(\abs{Dh_2}^2+\abs{Dh_1}^2\right) \le  4 \sum_{\alpha, \beta =1}^\infty    \int_{\mathbb K^{\alpha, \beta}}  \abs{Dh_1}^2 \\&\le  4 \sum_{\alpha, \beta =1}^\infty \frac{\epsilon^2}{2^{\alpha+\beta}} =4{\epsilon^2}
 \end{split}
 \end{equation*}
 Thus
 \[\norm{h_2-h_1}_{\Ar(\Omega)} \le \epsilon + 2 \epsilon = 3 \epsilon\]

The boundary of $\mathbb K^{\alpha, \beta}$ consists of two $\mathscr C^\infty$-smooth arcs $\Gamma^{\alpha, \beta}$ and $\Gamma^{\beta, \alpha}$ which share common endpoints, called  the apices of $\mathbb K^{\alpha, \beta}$. These are preimages of $\gamma^{\alpha, \beta}$ and $\gamma^{\beta, \alpha}$ under the mapping $h_1$, respectively.
 Outside of the apices, the homeomorphism $h_2\colon \overline{\mathbb K^{\alpha,\beta}}\onto \overline{\mathbb L^{\alpha,\beta}}$ is $C^{\infty}$ smooth with positive Jacobian. The smoothness is a classical result of Kellogg; {\it  a harmonic function  with $\mathscr C^\infty$-smooth  values on a smooth part of the boundary  is $\mathscr C^\infty$-smooth up to this part of the boundary}~\cite[Theorem 6.19]{GTb}. The positivity of the Jacobian on such part of the boundary follows from the convexity of its image, 
 see~\cite[p. 116]{Dub}.
 
In conclusion, $h_2$ is locally bi-Lipschitz in $\Omega \setminus h^{-1}(\mathcal V)$. The exceptional set $h^{-1}(\mathcal V)$ is discrete because $\mathcal V$ is.\\
\textsf{Finite energy case.} By~\eqref{36p} we have
 \begin{equation}\label{39}
\norm{Dh_2}_{\mathscr L^2(\Omega)} \le \norm{Dh_1}_{\mathscr L^2(\Omega)} \le \norm{Dh}_{\mathscr L^2(\Omega)}  - \delta
 \end{equation}

{\bf Step 3.} First we cover the set of vertices $\mathcal V$ by disks $\{\mathbb D_v \colon v \in \mathcal V\}$   centered at $v$ with radii small enough so that 
\begin{equation}\label{ddiam}
\diam \mathbb D_v\le \epsilon,
\end{equation}
and $\{ 3\mathbb D_v \colon v \in \mathcal V \}$ is a disjoint collection of disks in $\Omega^\ast$. Moreover,  their preimages under $h_2$ must satisfy
\begin{equation}\label{star}
\sum_{v \in \mathcal V} \int_{h_2^{-1}(3 \D_\nu)} \abs{Dh_2}^2 < {\epsilon}^2
\end{equation}
Denote by $\tilde{\Omega}^\ast = \Omega^\ast \setminus \bigcup_{v\in \mathcal V} \overline{\D}_v$ and $\tilde{\Omega}= \Omega \setminus \bigcup_{v \in \mathcal V} h_2^{-1} (\overline{\D}_v)$. Our focus for a while will be on one of the circular sides of  a lens $\mathbb L^{\alpha, \beta}$, say
\[\gamma^{\alpha, \beta}= \Q^\alpha \cap \partial \mathbb L^{\alpha, \beta} \subset \Q^\alpha\]
We truncate it near the endpoints by setting $\tilde{\gamma}^{\alpha, \beta}= \tilde{\Omega} \cap \gamma^{\alpha, \beta}$. Such truncated open  arcs are mutually disjoint; even more, their closures are isolated continua in $\Omega^\ast$. This means that there are disjoint neighborhoods of them. We are actually interested in a neighborhood of $\tilde{\gamma}^{\alpha, \beta}$ of the shape of a thin {\it concavo-convex lens} that we shall denote by $\tilde{\mathbb L}^{\alpha, \beta}$. By definition,  $\tilde{\gamma}^{\alpha, \beta} \subset \tilde{\mathbb L}^{\alpha, \beta} \subset \Q^\alpha$. The construction of such lens goes as follows. Let $a$ and $b$ denote the endpoints of $\tilde{\gamma}^{\alpha, \beta}$, we assemble two circular arcs $\tilde{\gamma}_+^{\alpha, \beta}$ and $\tilde{\gamma}_-^{\alpha, \beta}$ with endpoints at $a$ and $b$ to form together with their endpoints  a concavo-convex Jordan curve. This Jordan curve constitutes  the boundary of a circular lens  $\tilde{\mathbb L}^{\alpha, \beta}$. The term concavo-convex lens refers to the configuration in which  $\tilde{\mathbb L}^{\alpha, \beta}$ lies in the concave side of the arc  $\tilde{\gamma}_-^{\alpha, \beta}$ and convex side of $\tilde{\gamma}_+^{\alpha, \beta}$. It is clear that such lenses can be made arbitrarily thin so that 
 $\tilde{\mathbb L}^{\alpha, \beta} \subset \tilde{\Omega}^\ast$ and the closures of  $\tilde{\mathbb L}^{\alpha, \beta}$ will  still be isolated continua in $\Omega^\ast$. From now on we fix the family $\{\tilde{\mathbb L}^{\alpha, \beta} \colon \alpha \ne \beta \}$ of such concavo-convex lenses associated with the arcs $\tilde{\gamma}^{\alpha, \beta}$. We then look at  their preimages $\U^{\alpha, \beta} = h_2^{-1} ( \tilde{\mathbb L}^{\alpha, \beta})$ and the $\mathscr C^\infty$-smooth arcs $\Upsilon^{\alpha, \beta} =  h_2^{-1} ( \tilde{\gamma}^{\alpha, \beta})$. The endpoints of $\Upsilon^{\alpha, \beta}$ lie in $\partial \U^{\alpha, \beta}$. Moreover, $\Upsilon^{\alpha, \beta}$ splits $\U^{\alpha, \beta}$ into two disjoint subdomains $\U_+^{\alpha, \beta}$ and $\U_-^{\alpha, \beta}$ such that $\U^{\alpha, \beta} \setminus \Upsilon^{\alpha, \beta} = \U_+^{\alpha, \beta} \cup \U_-^{\alpha, \beta}$. Here we have a homeomorphism $h_2\colon \U^{\alpha, \beta} \onto  \tilde{\mathbb L}^{\alpha, \beta} $ which is $\mathscr C^\infty$-diffeomorphism on $\overline{\U}_+^{\alpha, \beta}$ and  $\mathscr C^\infty$-diffeomorphism on $\overline{\U}_-^{\alpha, \beta}$. Therefore, for some positive number $M_{\alpha, \beta}$, we have pointwise inequlities $\abs{Dh_2} \le {M_{\alpha, \beta}}$ and $\det Dh_2 \ge \frac{1}{M_{\alpha, \beta}}$ in both $\overline{\U}_+^{\alpha, \beta}$ and $\overline{\U}_-^{\alpha, \beta}$. Having established such a deformation of lenses and their preimages under $h_2$, we apply Corollary~\ref{muncor}. We infer that there is also a constant $M'_{\alpha, \beta}>0$ with the following property: to every neighborhood of $\Upsilon^{\alpha, \beta}$, say an open connected set  $\U_\circ^{\alpha, \beta}\subset \U^{\alpha, \beta}$ that contains $\Upsilon^{\alpha, \beta}$, there corresponds a $\mathscr C^\infty$-diffeomorphism, denoted by $h_3 \colon \U^{\alpha, \beta} \onto \tilde{\mathbb L}^{\alpha, \beta}$, such that
 \begin{equation}
 \begin{cases}
 h_3(z)=h_2(z) \quad \mbox{ for } z\in \U^{\alpha, \beta}\setminus {\U}_\circ^{\alpha, \beta}\\
 \abs{Dh_3} \le {M'_{\alpha, \beta}}\quad \mbox{ and } \quad \det Dh_3 \ge \frac{1}{M'_{\alpha, \beta}} \quad \mbox{in } \U^{\alpha, \beta}
 \end{cases}
 \end{equation}
We emphasize that $M'_{\alpha, \beta}$ is independent of the neighborhood ${\U}_\circ^{\alpha, \beta}$. We choose and fix ${\U}_\circ^{\alpha, \beta}$  thin enough to satisfy
\begin{itemize}
\item  $\overline{\U}_\circ^{\alpha, \beta} \subset {\U}^{\alpha, \beta} \cup \overline{\Upsilon}^{\alpha, \beta}$
\item $\abs{{\U}_\circ^{\alpha, \beta}} \le [M_{\alpha, \beta}+M'_{\alpha, \beta}]^{-2} \epsilon^2\,  2^{-\alpha - \beta}$
\item $\sup_{\U^{\alpha, \beta}_\circ} \abs{{Dh_2}} \le M_{\alpha, \beta}$
\item \textsf{in the finite energy case}, we also assume that
$\abs{{\U}_\circ^{\alpha, \beta}} \le [M'_{\alpha, \beta}]^{-2}\delta^2\,  4^{-\alpha - \beta-1}$
\end{itemize}
Recall that $\delta$ was defined by~\eqref{36p} and later appeared in~\eqref{39}.
This is certainly possible; for instance, take ${\U}_\circ^{\alpha, \beta}$ to be the preimage under $h_2$ of a sufficiently thin concavo-convex lens containing $\tilde{\gamma}^{\alpha, \beta}$. We call  $h_3 \colon \U^{\alpha, \beta} \onto \tilde{\mathbb L}^{\alpha, \beta}$  a smoothing of  $h_2 \colon \U^{\alpha, \beta} \onto \tilde{\mathbb L}^{\alpha, \beta}$  
associated with a given arc $\Upsilon^{\alpha, \beta}= h_2^{-1}(\tilde{\gamma}^{\alpha, \beta})$.
We now define a homeomorphism $h_3\colon \Omega\onto\Omega^*$ by the rule
\[
h_3=\begin{cases} \text{smoothing of $h_2$} \quad &\text{in $\U^{\alpha,\beta}$} \\
h_2 & \text{in $\Omega\setminus \bigcup_{\alpha,\beta}\U^{\alpha,\beta}$}
\end{cases}
\]
It belongs to $h_2+\Ar_\circ (\Omega)$. Obviously $h_3$ is a $\mathscr C^{\infty}$-diffeomorphism in $\widetilde{\Omega}$. We have for every $z\in\Omega$
\[
\begin{split}
\abs{h_3(z)-h_2(z)}&\le \begin{cases} \diam \widetilde{\mathbb L}^{\alpha,\beta} \quad &\text{for } z\in\U^{\alpha,\beta} \\
0 & \text{otherwise}
\end{cases} 
\\ &\le \diam \Q^\alpha \le {\epsilon}
\end{split}
\]
see~\eqref{doeq}. Hence $\norm{h_3-h_2}_{\mathscr C({\Omega})}\le{\epsilon}$. As regards  the energy of $h_3-h_2$ we find that
\begin{equation}\label{step3chain}
\begin{split}
\E_\Omega[h_3-h_2] &= \sum_{\alpha,\beta} \int_{{\U}_\circ^{\alpha,\beta}}\abs{Dh_3-Dh_2}^2
\\
&\le \sum_{\alpha,\beta} \abs{{\U}_\circ^{\alpha,\beta}}\sup_{{\U}_\circ^{\alpha,\beta}}
\left(\abs{Dh_3}+\abs{Dh_2}\right)^2
\\
& \le \sum_{\alpha,\beta} \abs{{\U}_\circ^{\alpha,\beta}} [M_{\alpha,\beta}'+M_{\alpha, \beta}]^2\le \sum_{\alpha, \beta =1}^\infty \frac{\epsilon^2}{2^{\alpha+ \beta}} \le {\epsilon^2} 
\end{split}
\end{equation}
These estimates sum up to
\[\norm{h_3-h_2}_{\Ar(\Omega)} \le {\epsilon}+ {\epsilon}= 2 \epsilon\]
Let us record for subsequent use the following estimate, obtained from~\eqref{star} and~\eqref{step3chain}.
\begin{align}\label{mes}
\sum_{v \in \mathcal V} \int_{h_3^{-1}(3 \D_\nu)} \abs{Dh_3}^2  &\le \sum_{v \in \mathcal V}\left( \int_{h_3^{-1}(3 \D_\nu) \setminus h_2^{-1}(3 \D_\nu)  } \abs{Dh_3}^2 +\int_{h_2^{-1}(3 \D_\nu)} \abs{Dh_3}^2  \right) \nonumber \\
\le \int_{\{h_3\ne h_2\}} \abs{Dh_3}^2 &+ 2\sum_{v \in \mathcal V} \left(\int_{h_2^{-1}(3\D_v)} \abs{Dh_3-Dh_2}^2 +   \int_{h_2^{-1}(3 \D_\nu)} \abs{Dh_2}^2    \right) \nonumber\\
&\le   \sum_{\alpha, \beta} \left(M'_{\alpha, \beta} \right)^2 \abs{\U^{\alpha, \beta}_\circ} + 2\epsilon^2 + 2 \epsilon^2 \le 5 \epsilon^2
\end{align}
\textsf{Finite energy case.} 
For the energy of  $h_3$, we observe that
\[
\begin{split}
\norm{Dh_3}_{\mathscr L^2(\Omega)} & \le \norm{Dh_3}_{\mathscr L^2(\Omega\setminus \cup \U_\circ^{\alpha, \beta})} + \sum_{\alpha, \beta} \norm{Dh_3}_{\mathscr L^2(\U_\circ^{\alpha, \beta})} \\
&=   \norm{Dh_2}_{\mathscr L^2(\Omega\setminus \cup \U_\circ^{\alpha, \beta})} +  \sum_{\alpha, \beta} \norm{Dh_3}_{\mathscr L^2(\U_\circ^{\alpha, \beta})} \\
&\le \norm{Dh_2}_{\mathscr L^2(\Omega)} + \sum_{\alpha, \beta} \abs{\U_\circ^{\alpha, \beta}}^{1/2}  \sup_{\U^{\alpha, \beta}_\circ} \abs{Dh_3}\\
&\le   \norm{Dh}_{\mathscr L^2(\Omega)} - \delta + \frac{\delta}{2}  
\end{split}
\]
Thus
\begin{equation}\label{lastadd}
\norm{Dh_3}_{\mathscr L^2(\Omega)} \le  \norm{Dh}_{\mathscr L^2(\Omega)} -  \frac{\delta}{2}  
\end{equation}

{\bf Step 4.} We have already upgraded the mapping $h$ to a homeomorphism $h_3 \colon \Omega \to \Omega^\ast$ such that $h_3\in h+\Ar_\circ(\Omega)$ and
\begin{equation}\label{sstep4}
\begin{split}
\norm{h_3-h}_{\Ar(\Omega)} & \le \norm{h_3-h_2}_{\Ar(\Omega)} + \norm{h_2-h_1}_{\Ar(\Omega)} + \norm{h_1-h}_{\Ar(\Omega)}\\ & \le 2 \epsilon + 3 \epsilon + 2\epsilon = 7 \epsilon
\end{split}
\end{equation}
Moreover, $h_3$ is a $\mathscr C^\infty$-diffeomorphism on $\Omega \setminus \bigcup_{v\in \mathcal V} h_2^{-1} (\overline{\D}_v)$. We now define a homeomorphism $h_4 \colon \Omega \onto \Omega^\ast$ by performing harmonic replacement of $h_3$ on each set $h_3^{-1}(2 \D_v)$. This gives us a  $\mathscr C^\infty$-diffeomorphism $h_4 \colon h_3^{-1} (2 \overline{\D}_v)\to 2 \overline{\D}_v$, see Step 2 for details. For each $z\in \Omega$
\[\abs{h_4(z)-h_3(z)} \le \begin{cases} 2 \diam \D_v \quad & \mbox{ if } z\in h_3^{-1} (2\D_v) \\
0 & \mbox{ otherwise}\end{cases} \quad \le 2{\epsilon}\]
Hence $\norm{h_4-h_3}_{\mathscr C (\Omega)} \le 2{\epsilon}$. Using~\eqref{mes} we estimate the energy as follows.
\[
\begin{split}
\E_\Omega [h_4-h_3] &= \sum_{v\in \mathcal V} \E_{h_3^{-1}(2\D_v)}[h_4-h_3]\\
& \le 2 \sum_{v\in \mathcal V} \left( \E_{h_3^{-1}(2\D_v)}[h_4] + \E_{h_3^{-1}(2\D_v)}[h_3]  \right)\\
& \le 4  \sum_{v\in \mathcal V} \E_{h_3^{-1}(2\D_v)}[h_3] \le 20\epsilon^2
\end{split}\]
Thus, by~\eqref{mes}
\begin{equation}\label{sstep5} \norm{h_4-h_3}_{\Ar(\Omega)} \le \epsilon + \sqrt{20} \epsilon \le 6 \epsilon\end{equation}
\textsf{Finite energy case.} By virtue of Dirichlet's principle and~\eqref{lastadd} we have
\begin{equation}\label{wearehungry}
\norm{Dh_4}_{\mathscr L^2(\Omega)} \le \norm{Dh_3}_{\mathscr L^2(\Omega)} \le \norm{Dh}_{\mathscr L^2(\Omega)}  - 
\frac{\delta}{2}
\end{equation}

{\bf Step 5.} The final step consists of smoothing $h_4$ in a neighborhood of each smooth Jordan curves $C_v= \partial h_3^{-1}(2\D_v)$. We proceed in much the same way as in Step 3, but we appeal to Corollary~\ref{muncircor} instead of Corollary~\ref{muncor}. By smoothing $h_4$ in a sufficiently thin neighborhood of  each $C_v$  we obtain a $\mathscr C^{\infty}$-diffeomorphism $h_5 \colon \Omega \onto \Omega^\ast$, $h_5\in h_4 + \Ar_\circ (\Omega)$ such that
\begin{equation}\label{needsharpen}
\norm{h_5-h_4}_{\Ar(\Omega)} \le {\epsilon}
\end{equation}
We now recapitulate the estimates~\eqref{sstep4},~\eqref{sstep5} and~\eqref{needsharpen} to obtain a $\mathscr C^\infty$-diffeomorphism in $\Omega$
\[H:=h_5\in h+\Ar(\Omega)\]
such that 
\[
\begin{split}
\norm{H-h}_{\Ar(\Omega)} & \le \norm{h_5-h_4}_{\Ar(\Omega)} +  \norm{h_4-h_3}_{\Ar(\Omega)} +  \norm{h_3-h}_{\Ar(\Omega)}\\
& \le \epsilon + 6 \epsilon + 7 \epsilon  =14 \epsilon     
\end{split}
\]
which is as strong as (ii) in Theorem~\ref{thmapprox}.\\
\textsf{Finite energy case.} To obtain the desired energy estimate $\E_\Omega [h_5] \le \E_\Omega [h]$, 
we need to sharpen the energy part in~\eqref{needsharpen}. By narrowing further the neighborhoods of $C_v$ we can be make the energy $\E_\Omega[h_5-h_4]$ as small as we wish; for example to obtain
\[\norm{Dh_5-Dh_4}_{\mathscr L^2 (\Omega)}< \frac{\delta}{2}\]
This is enough to conclude that
\[\norm{DH}_{\mathscr L^2(\Omega)} \le \norm{Dh}_{\mathscr L^2(\Omega)} \]
because of~\eqref{wearehungry}.
 \end{proof}

\section{Hopf differentials, Theorem~\ref{thmhopf}}
A quadratic differential on a domain $\Omega$ in the complex plane $\C$ takes the form $Q=F(z)\, \dtext z \otimes \dtext z$, where $F$ is a complex function on $\Omega$. Given a conformal change of the variable $z$, $z= \varphi(\xi)$, where $\varphi \colon \Omega' \to \Omega$, the pull back 
\[\varphi^\sharp (Q)= F\big(\varphi(\xi)\big)\, \dtext \varphi \otimes \dtext \varphi = F \big( \varphi (\xi) \big) \dot{\varphi}^2 (\xi) \, \dtext \xi \otimes \dtext \xi\]  
defines a quadratic differential on $\Omega'$. It is plain that for a complex harmonic function $h \colon \Omega \to \C$ the associated Hopf differential
\[Q_h= h_z \overline{h_{\bar z}}\,  \dtext z \otimes \dtext z\]
is holomorphic, meaning that
\begin{equation}\label{heq0}
\frac{\partial}{\partial \bar z} \left(h_z \overline{h_{\bar z}} \right)=0
\end{equation}
Conversely, if a Hopf differential $Q_h=h_z \overline{h_{\bar z}}\,  \dtext z \otimes \dtext z$ is holomorphic for some $\mathscr C^1$-mapping $h$, then $h$ is harmonic at the points where the Jacobian determinant $J(z,h):=\det Dh= \abs{h_z}^2- \abs{h_{\bar z}}^2 \ne 0$, see~\cite[10.5]{EL1} and our Remark~\ref{notredable}. Here the assumption that $J(z,h) \ne 0$ is critical. Let us illustrate it by the following.

\begin{example}\label{ex}
Consider a mapping $h\in \mathscr C^{1,1} (\C_\circ)$ defined on the punctured plane $\C_\circ = \C \setminus \{0\}$ by the rule
\begin{equation}\label{heq1}
h(z)= \begin{cases} \frac{z}{|z|}& \mbox{ for } 0< \abs{z} \le 1 \\
\frac{1}{2} \left(z+ \frac{1}{\bar z}\right)\qquad  &  \mbox{ for } 1 \le\abs{z} < \infty \end{cases}
\end{equation}
Direct computation shows that
\[
h_z (z)= \begin{cases} \frac{1}{2}|z|^{-1}& \mbox{ for } 0< \abs{z} \le 1 \\
\frac{1}{2} \qquad  &  \mbox{ for } 1 \le\abs{z} < \infty \end{cases}
\]
and 
\[
h_{\bar z} (z)= \begin{cases} -\frac{1}{2}|z| {\bar z}^{-2}& \mbox{ for } 0< \abs{z} \le 1 \\
-\frac{1}{2}\bar z^{-2} \qquad  &  \mbox{ for } 1 \le\abs{z} < \infty \end{cases}
\]
Thus 
\begin{equation}\label{tristars}
Q_h = - \frac{\dtext z \otimes \dtext z}{4z^2} \qquad \mbox{ in } \C_\circ
\end{equation}
It may be worth mentioning that the mapping $h$ in~\eqref{heq1} is the unique (up to rotation of $z$) minimizer of the Dirichlet energy
\[\mathcal E [H]= \int_{\A} \abs{DH}^2\]
over the annulus $\A= A(r,R)=\{z \colon r< \abs{z} < R  \}$, $0<r<1<R$, subject to all weak limits of homeomorphisms $H \colon \A \onto A(1, R_\ast)$, where $R_\ast = \frac{1}{2} \left(R + \frac{1}{R}\right)$, see~\cite{AIM}. Note that the Hopf differential of~\eqref{tristars} is real along the boundary circles of $\A$. The concentric circles are horizontal trajectories of $Q_h$. In fact this is a general property of minimizers~\cite[Lemma 1.2.5]{Job}. The general pattern is that with the loss of injectivity comes the loss of the Lagrange-Euler equation for the extremal mapping.
\end{example}

Properties of the function $h$ with holomorphic Hopf differential $Q= h_z \overline{h_{\bar z}} \, \dtext z \otimes \dtext z$ are of   interest in the studies of harmonic mappings~\cite{EL2, EL3, Job, Sc, Se}, minimal surfaces~\cite{DHKWb, Stb} and Teichm\"uller theory~\cite{GLb}. In this section we prove Theorem~\ref{thmhopf} which imposes fairly minimal assumptions that imply harmonicity of $\mathscr W^{1,2}$-solution to the equation~\eqref{heq0}. Some elements of the proof go back to~\cite{RS, RW}.


\begin{proof}[Proof of Theorem~\ref{thmhopf}]
As a consequence of the Stoilow factorization theorem~\cite[p. 56]{AIM} the branch set of $h$ is discrete, hence removable for continuous harmonic functions. Thus we assume that  $h \colon \Omega \onto \Omega^\ast$ is a homeomorphism of Sobolev class $\W^{1,2}_{\loc} (\Omega, \Omega^\ast)$ such that
\begin{equation}
h_z \overline{h_{\bar z}} =F(z) \quad \mbox{ is holomorphic in } \; \Omega
\end{equation}
By virtue of Theorem~\ref{thmapprox}, there exists a sequence of diffeomorphisms $h^j \colon \Omega \onto \Omega^\ast$ converging $c$-uniformly and strongly in  $\W^{1,2}_{\loc} (\Omega, \Omega^\ast)$  to $h$. Denote by
\begin{equation}
h_z^j h^j_{\bar z}=:F^j \in \mathscr L_{\loc}^1(\Omega) 
\end{equation}
Thus $F^j \to F$ strongly in $ \mathscr L_{\loc}^1(\Omega)$. Let us first dispose of an easy case.\\
{\bf Case 0.} {\it The homogeneous equation $F \equiv 0$}. Since $h^j$ are diffeomorphisms the Jacobian determinant $J(z,h^j)= \abs{h_z^j}^2-\abs{h_{\bar z}^j}^2$ is either positive everywhere in $\Omega$ or negative everywhere in $\Omega$. Let us settle the case when $J(z,h^j)>0$ for infinitely many indices $j=1,2,\dots$.
For such $j$ we have $\abs{h^j_z}>\abs{h^j_{\bar z}}$, which yields 
$\abs{h^j_{\bar z}}^2\le \abs{h^j_z h^j_{\bar z}}$. Passing to the $\mathscr L^1$-limit we obtain 
\[\abs{h_{\bar z}}^2\le \abs{h_z h_{\bar z}} = \abs{F(z)}\equiv 0.\]
Thus $h$ is holomorphic, by  Weyl's lemma. Similarly, in case $J(z,h^j)<0$ for infinitely many indices $j=1,2,\dots$,
we find that $h$ is antiholomorphic. 

\begin{remark}\label{rem} We observe, based on the above arguments, that for this homogeneous equation $h_z\overline{h_{\bar z}}\equiv 0$
every solution $h\in \W_{\loc}^{1,2}(\Omega)$ obtained as the weak $\W^{1,2}$-limit of homeomorphisms is either holomorphic or 
antiholomorphic. The situation is dramatically different if $h_z\overline{h_{\bar z}}\not\equiv 0$; some topological assumption on $h$ are necessary, as illustrated in Example~\ref{ex}. 
\end{remark}

{\bf Case 1.} {\it Nonhomogeneous equation $F\not\equiv 0$}. The function $F$, being holomorphic, may vanish only at isolated points. 
Since isolated points are removable for bounded harmonic functions, it suffices to consider the set where $F\ne 0$. 
Proceeding further in this direction, we may and do assume that $F(z)\equiv 1$ (by a conformal change of the $z$-variable) and $h$ is a 
$\W^{1,2}$-homeomorphism in the closure of the unit square $\Q=\{x+iy\colon 0<x<1, 0<y<1\}$.
The problem now reduces to establishing that the equation 
\begin{equation}\label{eeq3}
h_{z}\overline{h_{\bar z}}\equiv 1
\end{equation}
 implies $\Delta h=0$.  This will be proved indirectly by means of the energy-minimizing property  
\begin{equation}\label{ineq4}
\mathcal E_{\Q}[h]\le \mathcal E_{\Q}[H]
\end{equation}
where $H\colon \Q\to h(\Q)$ is any homeomorphism in $h+\Ar_\circ(\Q)$; in particular, $H=h$ on $\partial\Q$. 
Indeed, if $h$ were not harmonic, we would be able to decrease its energy by harmonic replacement (Propositions~\ref{prpoisson} and~\ref{RKC}), 
contradicting~\eqref{ineq4}.

\subsection{Proof of the inequality~\eqref{ineq4}}

With the aid of the approximation theorem we need only prove~\eqref{ineq4} for mappings $H\in h+ \Ar_\circ(\Q)$ that
are diffeomorphisms on $\Q$. From now on we assume that this is the case. Denote $\Q^*=h(\Q)=H(\Q)$. We consider a sequence $h^j\in h+\Ar_\circ(\Q)$
of diffeomorphisms $h^j\colon \Q\onto\Q^*$ converging in $\Ar(\Q)$ to $h$. Moreover we may also assume that $Dh^j\to Dh$ almost everywhere in $\Q$ by passing to a subsequence if necessary. Now the sequence $\chi^j\colon\overline{\Q}\to\overline{\Q}$ of self-homeomorphisms of the closed unit disk
given by $\chi^j=H^{-1}\circ h^j$, where $\chi^j=\id$ on $\partial\Q$, is converging uniformly on $\overline{\Q}$ to $\chi=H^{-1}\circ h$.
It is important to observe that $\chi\in \W_{\loc}^{1,2}(\Q)$ and $\chi^j$ converges to $\chi$ in $\W^{1,2}(\Q')$ on any compactly
contained subdomain $\Q'\Subset \Q$. Since $h^j$ and $(\chi^j)^{-1}$ are diffeomorphisms on $\Q'$ and $\chi^j(\Q')$, respectively,
the chain rule can be applied to the composition $H=h^j\circ (\chi^j)^{-1}$. For $w\in \chi^j(\Q')$ we have
\begin{equation*}
\begin{split}
\frac{\partial H(w)}{\partial w} &= h_z^j(z)\frac{\partial (\chi^j)^{-1}}{\partial w}+h_{\bar z}^j(z)\frac{\partial (\chi^j)^{-1}}{\partial \bar w} \\
\frac{\partial H(w)}{\partial \bar w} &= h_z^j(z)\frac{\partial (\chi^j)^{-1}}{\partial \bar w}+h_{\bar z}^j(z)\overline{\frac{\partial (\chi^j)^{-1}}{\partial w}}
\end{split}
\end{equation*}
where $z=(\chi^j)^{-1}(w)$.

The partial derivatives of $(\chi^j)^{-1}$ at $w$ can be expressed in terms of $\chi_{_z}^j(z)$ and $\chi_{\bar z}^j(z)$ by the rules
\begin{equation*}
\begin{split}
\frac{\partial (\chi^j)^{-1}}{\partial w} &= \frac{\chi_z^j(z)}{J(z,\chi^j)} \\
\frac{\partial (\chi^j)^{-1}}{\partial \bar w} &= - \frac{\chi_{\bar z}^j(z)}{J(z,\chi^j)}
\end{split}
\end{equation*}
where the Jacobian determinant $J(z,\chi^j)$ is strictly positive. This yields
\begin{equation*}
\begin{split}
\frac{\partial H}{\partial w} & =\frac{h_z^j\overline{\chi_z^j}-h_{\bar z}^j\overline{\chi_{\bar z}^j}}{J(z,\chi^j)} \\
\frac{\partial H}{\partial \bar w} & =\frac{h_{\bar z}^j{\chi_z^j}-h_{\bar z}^j{\chi_{\bar z}^j}}{J(z,\chi^j)}
\end{split}
\end{equation*}
We compute the energy integral of $H$ over the set $\chi^j(\Q')$ by substitution $w=\chi^j(z)$,
\begin{equation*}
\begin{split}
\mathcal E_{\chi^j(\Q')}[H]&=2\int_{\chi^j(\Q')} \left(\abs{H_w}^2+\abs{H_{\bar w}}^2\right)\,\dtext w \\
&=2\int_{\Q'}\frac{\abs{h_z^j \overline{\chi_z^j} - h_{\bar z}^j \overline{\chi_{\bar z}^j}}^2+
\abs{h_{\bar z}^j \chi_z^j - h_{z}^j \chi_{z}^j}^2}{\abs{\chi_z^j}^2-\abs{\chi_{\bar z}^j}^2}\,\dtext z
\end{split}
\end{equation*}
On the other hand, the energy of $h^j$ over the set $\Q'$ is
\[
\mathcal E_{\Q'}[h^j]=2\int_{\Q'}\left(\abs{h_z^j}^2+\abs{h_{\bar z}^j}^2\right)\,\dtext z
\]
Subtract these two integrals to obtain
\begin{equation}\label{cchain}
\begin{split}
\mathcal E_{\Q}[H]-\mathcal E_{\Q'}[h^j] &\ge \mathcal E_{\chi^j(\Q')}[H]-\mathcal E_{\Q'}[h^j] \\
&=  4\int_{\Q'}  \frac{\left(\abs{h_z^j}^2+\abs{h_{\bar z}^j}^2\right)\cdot \abs{\chi_{\bar z}^j}^2
-2\re \left[h_z^j\overline{h_{\bar z}^j} \overline{\chi_z^j}\chi_{\bar z}^j\right]}{\abs{\chi_z^j}^2-\abs{\chi_{\bar z}^j}^2}\,\dtext z \\
& \ge 4 \int_{\Q'} \frac{ 2\abs{h_z^jh_{\bar z}^j} \abs{\chi_{\bar z}^j}^2
-2\re \left[h_z^j\overline{h_{\bar z}^j} \overline{\chi_z^j}\chi_{\bar z}^j\right]}{\abs{\chi_z^j}^2-\abs{\chi_{\bar z}^j}^2}\,\dtext z \\
& = 4 \int_{\Q'} \left[\frac{\abs{\chi_z^j-\sigma^j(z) \chi_{\bar z}^j}^2}{\abs{\chi_z^j}^2-\abs{\chi_{\bar z}^j}^2} -1 \right]
\, \abs{h_z^jh_{\bar z}^j} \,\dtext z
\end{split}
\end{equation}
where we have introduced the notation
\[
\sigma^j = \sigma^j(z) = \begin{cases}
{h_z^j\overline{h_{\bar z}^j}}{\, \abs{h_z^jh_{\bar z}^j}^{-1}} \qquad &\text{if } h_z^jh_{\bar z}^j\ne 0 \\
1 & \text{otherwise.}
 \end{cases}
\]
Note that $\abs{\sigma^j}=1$ and $\sigma^j\to 1$ almost everywhere.

Upon using H\"older's inequality we continue the chain~\eqref{cchain} as follows.
\begin{equation}\label{cchain2}
\ge \; 4\frac{\left[\int_{\Q'}\left| \chi_z^j-\sigma^j\chi_{\bar z}^j\right| \,\sqrt{\abs{h_z^jh_{\bar z}^j}}\,\dtext z\right]^2}{\int_{\Q'} J(z,h^j)\,\dtext z}
-4\int_{\Q'} \abs{h_z^jh_{\bar z}^j}.
\end{equation}
The denominator in~\eqref{cchain2} is at most $1$ because
\[\int_{\Q'} J(z,h^j)\,\dtext z = \abs{\chi^j(\Q')}\le\abs{\Q}=1.\]
Therefore,
\begin{equation*}
\mathcal E_{\Q}[H]-\mathcal E_{\Q'}[h^j]\ge 4\left[\int_{\Q'}\left| \chi_z^j-\sigma^j\chi_{\bar z}^j\right|\,\sqrt{\abs{h_z^jh_{\bar z}^j}}\,\dtext z  \right]^2 -4\int_{\Q'}\abs{h_z^jh_{\bar z}^j}\, \dtext z.
\end{equation*}
It is at this point that we can pass to the limit as $j\to\infty$, to obtain
\begin{equation}\label{cchain3}
\mathcal E_{\Q}[H]-\mathcal E_{\Q'}[h]\ge 4\left[\int_{\Q'}\abs{\chi_z-\chi_{\bar z}}\, \dtext z \right]^2-4\abs{\Q'}.
\end{equation}
Since $\Q'$ was an arbitrary compactly contained subdomain of $\Q$, the estimate~\eqref{cchain3} remains valid with $\Q'$ replaced by $\Q$.
\begin{equation}\label{cchain4}
\begin{split}
\mathcal E_{\Q}[H]-\mathcal E_{\Q}[h] &\ge 4\left[\int_{\Q}\left|\frac{\partial\chi}{\partial y}\right|\,\dtext x\,\dtext y\right]^2-4  \\
&\ge 4\int_0^1 \left|\int_0^1 \frac{\partial\chi (x,y)}{\partial y}\,\dtext y\right|\,\dtext x - 4 \\
&= 4 \int_0^1 \abs{\chi(x,1)-\chi(x,0)}\,\dtext x-4 =4-4 =0
\end{split}
\end{equation}
as desired.
\end{proof}

\begin{remark}\label{notredable}
When specialized to the case $h \in \mathscr C^1$, Theorem~\ref{thmhopf} shows that $h$ is harmonic outside of the zero set of its Jacobian.
\end{remark}

\section{Auxiliary smoothing results}

Here we present some results concerning smoothing of piecewise differentiable planar homeomorphisms.
They can be found in~\cite{Mu} in greater generality, but since we require quantitative control of derivatives,
a self-contained proof is in order. Here it is more convenient to use the operator norm of a matrix, denoted by $\norm{\cdot}$. Note that $\norm{A} \le \abs{A} \le 2 \norm{A}$ for $2 \times 2$-matrices.

\begin{proposition}\label{mun} Let $\U\subset\R^2$ be a  domain containing an open segment $I$ with endpoints on the boundary $\partial \U$ which splits $\U$ into two subdomains $\U_1$ and $\U_2$ such that $\U \setminus I= \U_1\cup \U_2$.
Suppose that $f\colon \overline{\U}\onto \overline{\U^*}\subset\R^2$ is a homeomorphism with the following properties:
\begin{enumerate}[(i)]
\item\label{mun2a} For $j=1,2, \dots$ the restriction of $f$ to $\overline{\U_j} $ is $\mathscr C^{\infty}$-smooth, equals the identity on $I$;
\item\label{mun2} There is a constant $M>0$ such that for $j=1,2$ the restriction of $f$ to $\overline{\U_j}$ satisfies $\norm{Df}\le M$ and $\det Df\ge M^{-1}$.
\end{enumerate}

Then for any open set
$\U_\circ$ with $I\subset \U_\circ\subset \U$ there is a $\mathscr C^{\infty}$-diffeomorphism $g\colon \U\to \U^*$ such that
\begin{itemize}
\item $g$ agrees with $f$ on $\U\setminus \U_\circ$ (and also on $I$);
\item $\norm{Dg}\le 20M$ and $\det Dg\ge (20M)^{-1}$ on $\U$.
\end{itemize}
\end{proposition}

\begin{proof} Without loss of generality $I\subset \R=\{(x,y) \colon y=0\}$. We write $f$ in components as $(u,v)$ where $u$ and $v$
are functions of $x$ and $y$. Let us introduce a notation; given any $\mathscr C^{\infty}$-smooth function $\beta\colon\R\to[0,\infty)$, denote 
$V(\beta)=\{(x,y)\in \R^2\colon \abs{y}<\beta(x)\}$. We can and do choose $\beta$ so that
$I\subset V(\beta)\subset \U_\circ$, and further scale it down until the following holds.
\begin{equation}\label{beta}
\begin{split}
\abs{\beta'(x)} &\le\frac{1}{40M} \qquad \mbox{ for all } x\in \R; \\
\abs{v_x}&\le \frac{1}{50M^2} \qquad \text{in }V(\beta)\setminus I, \quad \mbox{ because } v(x,0)=0; \\
\abs{u_x-1}&\le \frac{1}{10} \qquad \quad \;  \text{ in }V(\beta)\setminus I, \quad \mbox{ because } u(x,0)=x.
\end{split}
\end{equation}
As a consequence of ~\eqref{mun2} and~\eqref{beta},
\begin{equation}\label{beta1}
v_y \ge \frac{M^{-1}-\abs{u_y v_x}}{u_x}\ge \frac{1}{2M}.
\end{equation}
Since $v$ is also $M$-Lipschitz by~\eqref{mun2}, the following double inequality holds in $V(\beta)\setminus I$.
\begin{equation}\label{vest}
\frac{1}{2M} \le \frac{v}{y}\le M.
\end{equation}

Let us fix  be a nondecreasing $\mathscr C^{\infty}$ function $\alpha\colon\R\to\R$ such that $\alpha(t)=0$ for $t\le 1/3$. Let
$\alpha(t)=1$ for $t\ge 2/3$. Moreover, $\alpha'(t)\le 4$ for all $t\in\R$ and $\alpha(\infty)=1$, by convention.  Now we introduce a modification of $u$ on  $\U$ by setting
\[
\tilde u := \alpha(t)u+(1-\alpha(t))x \qquad \text{where } t=\begin{cases}\frac{\abs{y}}{\beta(x)} \quad & \mbox{ if } \beta(x) \ne 0 \\
\infty & \mbox{ otherwise} \end{cases}.
\]
Note that $\tilde u=u$ outside of $V(\beta)$. In $V(\beta)\setminus I$ we compute the derivatives
as follows.
\begin{equation}\label{d1}
\begin{split}
\tilde u_x &= - t^2  \alpha'(t) \beta'(x) \frac{u-x}{\abs{y}} + \alpha(t)u_x+1-\alpha(t) \\
\tilde u_y &= t\alpha'(t)\frac{u-x}{y} + \alpha(t) u_y
\end{split}
\end{equation}
Since $u$ is $M$-Lipschitz by~\eqref{mun2}, we have $\abs{u-x}\le M\abs{y}$. From this,~\eqref{beta}
and~\eqref{d1} we obtain
\begin{equation}\label{d2}
\frac{8}{10} \le \tilde u_x\le \frac{12}{10},
\quad \text{ and } \quad \abs{\tilde u_y}\le 5M,
\end{equation}
which combined with~\eqref{beta1} yields
\begin{equation}\label{d3}
\tilde u_x v_y-\tilde u_y v_x \ge \frac{8}{10}\frac{1}{2M}- \frac{5M}{50M^2}= \frac{3}{10 M}.
\end{equation}

Next we modify $v$ on $\U$. Specifically,
\[
\tilde v := \alpha(s) v+(1-\alpha(s))\frac{y}{2M} \qquad \text{where } s= \begin{cases} \frac{3\abs{y}}{\beta(x)} \quad & \mbox{ if } \beta(x) \ne 0 \\
\infty& \mbox{ otherwise}  \end{cases}
\]
Note that $\tilde{v}=v$ outside of $V(\beta/3)$, and on the set $V(\beta/3)$ we already have $\tilde{u} \equiv x$.

Computations similar to~\eqref{d1} yield (on the set $V(\beta/3)\setminus I$)
\begin{equation}\label{d1v}
\begin{split}
\tilde v_x &= -\frac{1}{3}\alpha'(s) s^2 \frac{v-y}{\abs{y}} + \alpha(s)v_x; \\
\tilde v_y &= \frac{s\alpha'(s)}{y}\left(v-\frac{y}{2M}\right) + \alpha(s)v_y+\frac{1-\alpha(s)}{2M}.
\end{split}
\end{equation}
Straightforward estimates based on~\eqref{beta},~\eqref{beta1} and~\eqref{vest} comply
\begin{equation}\label{d4}
\begin{split}
\abs{\tilde v_x}& \le \frac{4 M}{3}+\frac{1}{50M^2}<\frac{3M}{2}, \\
\frac{1}{2M}& \le \tilde v_y \le  5M.
\end{split}
\end{equation}

It remains to check that the mapping $g:=(\tilde u,\tilde v)$, which agrees with $f$ outside of $V(\beta)$, satisfies all
the requirements. As regards $\mathscr C^\infty$-smoothness we need only check it on $V(\beta/9)$. But in this neighborhood of $I$ we have a linear mapping, $g(x,y)= \left(x, \frac{y}{2M}\right)$, so $\mathscr C^\infty$-smooth. By virtue of~\eqref{d2} and~\eqref{d4} we have $\norm{Dg}\le 20M$. The desired lower bound for
$\det Dg$ follows from~\eqref{d3} and~\eqref{d4}. Consequently, $g$ is a local diffeomorphism, and since
it agrees with $f$ on $\partial V(\beta)$, it is in fact a diffeomorphism, by a topological result: {\it a local homeomorphism which shares boundary values with a homeomorphism is injective}~\cite[Lemma 8.2]{Mu}.
\end{proof}

We also need a polar version of Proposition~\ref{mun}. 

\begin{corollary}\label{muncir} Let $\U\subset\R^2$ be a domain containing a circle $\T$.
Suppose that $f\colon \U\onto \U^*\subset\R^2$ is a homeomorphism with the following properties:
\begin{enumerate}[(i)]
\item\label{munc1} The restriction of $f$ to $\T$ is the identity mapping;
\item\label{munc2} There is a constant $M>0$ such that the restriction of $f$ to either component of 
$\U\setminus \T$ is $\mathscr C^{\infty}$-smooth with $\norm{Df}\le M$ and $\det Df\ge M^{-1}$.
\end{enumerate}
Then for any open set
$W$ with $\T\subset W\subset \U$ there is a $\mathscr C^{\infty}$-diffeomorphism $g\colon \U\to \U^*$ such that
\begin{itemize}
\item $g$ agrees with $f$ on $\U\setminus W$ and on $\T$;
\item $\norm{Dg}\le 80M$ and $\det Dg\ge (80M)^{-1}$ on $\U$.
\end{itemize}
\end{corollary}

\begin{proof} It is convenient to identify $\R^2$ with $\C$. 
Without loss of generality $\T=\{z\in\C\colon \abs{z}=1\}$. Let $\psi(\zeta)=\exp(i\zeta)$. 
The mapping $F=\psi^{-1}\circ f\circ \psi$
is well-defined in some open horizontal strip $S_h=\{z\in\C\colon \abs{\im z}<\epsilon \}$ 
which we choose thin enough so that $\psi(S_\epsilon)\subset W$ and $\abs{\psi'}^2< e^{2\epsilon }\le 2$.
Note that $F$ is $2\pi$-periodic and satisfies
\[\norm{DF}\le 2M \quad \text{ and }\quad \det DF\ge (2M)^{-1}.\]

The proof of Proposition~\ref{mun} applies to $F$ with no changes other than one simplification: 
$\beta>0$ is now a small positive constant rather than a function. 
Thus we obtain a diffeomorphism $G$ which agrees with $F$ on
$\R\cup (S\setminus V(\beta))$ and satisfies $\norm{DG}\le 40$ and $\det DG\ge (40M)^{-1}$.
Since $F$ was $2\pi$-periodic, so is $G$. Thus, $g:=\psi\circ G\circ \psi^{-1}$ is the desired diffeomorphism.
\end{proof}

Our applications require slightly more general versions of Proposition~\ref{mun} and Corollary~\ref{muncir}, where the separating curve is allowed to have other shapes and $f$ is not required to agree with the identity on the curve. 

\begin{definition}\label{nice}
A parametric curve $\Gamma \colon (0,1) \to \R^2$ is {\it regular} if $\Gamma$ extends to a bigger interval 
$(a,b)\supset [0,1]$ so that the extended mapping is a $\mathscr C^\infty$-diffeomorphism onto its image. 
\end{definition}

Note that a regular curve $\Gamma$ has well-defined endpoints $\Gamma(0)$ and $\Gamma(1)$. Also,  
$\Gamma$ extends to an injective $\mathscr C^\infty$-mapping $\Phi\colon (0,1)\times (-1,1)\to\R^2$
such that $\norm{D\Phi}$ and $\norm{(D\Phi)^{-1}}$ are bounded. This follows from the existence of a tubular neighborhood
of the image of $\Gamma$~\cite[Theorem 4.26]{MRb}.  

Corollaries~\ref{muncor} and~\ref{muncircor}, given below, generalize Proposition~\ref{mun} and  Corollary~\ref{muncir} 
respectively. 

\begin{corollary}\label{muncor}  
Let $\U \subset \R^2$ be a domain containing the image of a regular arc $\Gamma$ with endpoints on the boundary $\partial \U$ which divides $\U$ into two subdomains $\U_1$ and $\U_2$ such that $\U \setminus \Gamma=\U_1 \cup \U_2$. Suppose $f \colon \overline{\U} \to \overline{\U^\ast} \subset \R^2$ is a homeomorphism such that $f \circ \Gamma$ is also regular and the restriction of $f$ to each $\overline{\U_i}$ is $\mathscr C^\infty$-smooth and satisfies
\[ \abs{Df(z)} \le M, \quad \det Df(z) \ge \frac{1}{M} \quad \mbox{for } z\in \U_i\]
where $M$ is a positive constant. Then there is a constant $M'>0$ such that to every open set $\U' \subset \U$ with $\Gamma \subset \U'$ there corresponds a $\mathscr C^\infty$-diffeomorphism $g\colon \U \onto \U^\ast$ with the following properties
\begin{itemize}
\item $g(z)=f(z)$ for $z\in \U \setminus \U'$ (and also on $\Gamma$)
\item $\abs{Dg(z)} \le M'$ and $\det Dg(z) \ge \frac{1}{M'}$ on $\U$.
\end{itemize}
\end{corollary}

\begin{proof}  Let $\Q=(0,1)\times (-1,1)$. Let $\Phi$ and $\Psi$ be the extensions of $\Gamma$
and $f\circ \Gamma$ to $\Q$ as in Definition~\ref{nice}. There is a domain $\widetilde{\U}$ such that 
$(0,1)\times \{0\}\subset \widetilde{\U} \subset \Q$, $\Phi(\widetilde{\U})\Subset \U'$, and 
the composition $F:=\Psi^{-1}\circ f\circ \Phi$ is defined in $\widetilde{\U}$. Note that $F=\id$ on $(0,1)\times \{0\}$. 
We apply Proposition~\ref{mun} (with  $\widetilde{\U}$ in place of $\U$ and with 
$F$ in place of $f$) and obtain a $\mathscr C^\infty$-diffeomorphism $G \colon \widetilde{\U} \to F(\widetilde{\U})$. Finally, replace $F$ within $\widetilde{\U}$ with the diffeomorphism $g=\Psi\circ G\circ \Phi^{-1}$. 
\end{proof}

\begin{corollary}\label{muncircor}  
Let $\U \subset \R^2$ be a domain containing the image of a $\mathscr C^\infty$-smooth Jordan curve $\Gamma$  which divides $\U$ into two subdomains $\U_1$ and $\U_2$ such that $\U \setminus \Gamma=\U_1 \cup \U_2$. Suppose $f \colon \overline{\U} \to \overline{\U^\ast} \subset \R^2$ is a homeomorphism such that the restriction of $f$ to each $\overline{\U_i}$ is $\mathscr C^\infty$-smooth and satisfies
\[ \abs{Df(z)} \le M, \quad \det Df(z) \ge \frac{1}{M} \quad \mbox{for } z\in \U_i\]
where $M$ is a positive constant. Then there is a constant $M'>0$ such that to every open set $\U' \subset \U$ with $\Gamma \subset \U'$ there corresponds a $\mathscr C^\infty$-diffeomorphism $g\colon \U \onto \U^\ast$ with the following properties
\begin{itemize}
\item $g(z)=f(z)$ for $z\in \U \setminus \U'$ (and also on $\Gamma$)
\item $\abs{Dg(z)} \le M'$ and $\det Dg(z) \ge \frac{1}{M'}$ on $\U$.
\end{itemize}
\end{corollary}

\begin{proof} The proof of Corollary~\ref{muncor} is easily adapted to this case. 
\end{proof}

\section{Concluding remarks}

One may wonder whether the proof of Theorem~\ref{thmapprox} can be extended to the spaces $\W^{1,p}$, $1<p<\infty$, by means of the $p$-harmonic replacement in place of Proposition~\ref{RKC}. Indeed, $p$-harmonic mappings are $\mathscr C^{1, \alpha}$-smooth~\cite{Ur}. However, the injectivity of $p$-harmonic replacement of a homeomorphism is unclear.
\begin{question}
Is there a version of the  Rad\'o-Kneser-Choquet theorem for $p$-harmonic mappings? That is, does the $p$-harmonic 
extension of a homeomorphism onto a convex Jordan curve  enjoy the injectivity property? 
\end{question}

An attempt to extend Theorem~\ref{thmapprox} to higher dimensions faces another obstacle: the  
Rad\'o-Kneser-Choquet theorem fails in dimensions $n\ge 3$ as was proved by Laugesen~\cite{La}.

\bibliographystyle{amsplain}

\end{document}